\documentclass[10 pt,a4paper,notitlepage,twoside,openright]{article}

\usepackage{amsmath}
\usepackage{amsfonts}
\usepackage{amssymb}
\usepackage{amsthm}
\usepackage{layout}
\usepackage{pifont}
\usepackage{graphicx}
\usepackage[all]{xy}
\usepackage{tikz}

\newtheorem{theorem}{Theorem}[section]
\newtheorem{corollary}[theorem]{Corollary}
\newtheorem{proposition}[theorem]{Proposition}

\theoremstyle{remark}
\newtheorem{remark}{Remark}[section]

\theoremstyle{plain}
\newtheorem{teorem}{Teorem}[section]
\newtheorem{lemma}[teorem]{Lemma}

\DeclareMathOperator{\Div}{div}

\DeclareMathOperator{\grad}{grad}
\DeclareMathOperator{\curl}{curl}
\DeclareMathOperator{\defo}{def}

\DeclareMathOperator{\Skew}{skew}
\DeclareMathOperator{\op}{op}

\DeclareMathOperator{\tr}{tr}

\DeclareMathOperator{\transp}{\textsc{t}}

\renewcommand{\div}{\Div}

\newcommand{\rmd}{\mathrm{d}}

\newcommand{\rmH}{\mathrm{H}}
\newcommand{\rmL}{\mathrm{L}}

\newcommand{\bbA}{\mathbb{A}}

\newcommand{\bbC}{\mathbb{C}}

\newcommand{\bbM}{\mathbb{M}}
\newcommand{\bbN}{\mathbb{N}}

\newcommand{\bbR}{\mathbb{R}}
\newcommand{\bbS}{\mathbb{S}}

\newcommand{\bbV}{\mathbb{V}}
\newcommand{\bbZ}{\mathbb{Z}}

\newcommand{\calO}{\mathcal{O}}

\newcommand{\calQ}{\mathcal{Q}}
\newcommand{\calR}{\mathcal{R}}

\newcommand{\calT}{\mathcal{T}}

\newcommand{\beq}{\begin{equation}}
\newcommand{\eeq}{\end{equation}}

\newcommand{\ts}{\textstyle}

\newcommand{\curltcurl}{\curl  \transp \curl}

\newcommand{\id}{\mathrm{id}}

\newcommand{\jump}[1]{[\![ #1 ]\!]}

\newcommand{\myand}{ \quad \textrm{and} \quad }
\newcommand{\myor}{ \quad \textrm{or} \quad }

\newcommand{\infsup}[4]{\inf_{#1}\sup_{#2}\frac{#3}{#4}}
\newcommand{\cleq}{\preccurlyeq}
\newcommand{\cgeq}{\succcurlyeq}
\newcommand{\ceq}{\approx}

\newcommand{\describe}[1]{{\bf (#1)}}

\title{\vspace*{-1.5cm}
\begin{flushright}
\begin{minipage}{4.7cm}
\tiny
Revised version of:\\
\mbox{}\\
{\sc Dept. of Math./CMA \hfill University of Oslo\\
Pure Mathematics \hfill     No~13\\
ISSN 0806--2439 \hfill     May 2008}
\end{minipage}
\end{flushright}
\vskip 1cm
On the linearization of Regge calculus}

\author{
 Snorre H. {\sc Christiansen}\footnote{CMA, University of Oslo, PO Box 1053 Blindern, NO-0316 Oslo, Norway. email : {\tt snorrec@math.uio.no}}
}

\begin{document}

\date{}

\maketitle

\begin{abstract}
We study the linearization of three dimensional Regge calculus around Euclidean metric. We provide an explicit formula for the corresponding quadratic form and relate it to the $\curltcurl$ operator which appears in the quadratic part of the Einstein-Hilbert action and also in the linear elasticity complex. We insert Regge metrics in a discrete version of this complex, equipped with densely defined and commuting interpolators. We show that the eigenpairs of the $\curltcurl$ operator, approximated using the quadratic part of the Regge action on Regge metrics, converge to their continuous counterparts, interpreting the computation as a non-conforming finite element method. 
\end{abstract}


\section{Introduction}
Regge calculus \cite{Reg61} is a combinatorial approach to Einstein's theory of general relativity \cite{Wal84}. Space-time is represented by a simplicial complex. Given this simplicial complex, a finite dimensional space of metrics is defined, each metric being determined by a  choice of edge lengths. We call such metrics Regge metrics. A functional, defined on this space of metrics and mimicking the Einstein-Hilbert action, is provided. We call this functional the Regge action. A critical point of the Regge action on the space of Regge metrics is generally believed to be a good approximation to a true solution of Einstein's equations \cite{MisThoWhe73}.

Regge calculus (RC) is quite popular in studies of quantum gravity \cite{RegWil00}. Its discrete nature also makes it a natural candidate for the construction of efficient algorithms to simulate the classical field equations, a possibility expressed already in the last sentence of Regge's paper. In this direction we are aware of, in particular \cite{Por87a}\cite{Por87b}\cite{BarEtAl97}\cite{GenMil98}\cite{Gen02}. However, it seems that the bulk of numerical relativity computations are performed using other methods. One difficulty with simulating Einstein's equations is the gauge freedom (diffeomorphism invariance) which imposes constraints on the evolution. Hyperbolicity in this context is a delicate matter \cite{Reu04}. Progress on the simulation of merging black holes \cite{Pre06}, seems to have been achieved in large part by judiciously choosing which partial differential equations to solve (in particular the gauge conditions), so that a number of traditional discretization philosophies, including finite difference, finite element and collocation methods,  are successfully applied today \cite{Alc08}\cite{GraNov09}\cite{BauSha10}, in support of the emergent field of gravitational wave astronomy \cite{SatSch09}.

We hope that this paper can contribute to developing RC into a good alternative, or facilitate the integration of some of its appealing features into currently used methods. Its geometric ``coordinate free'' nature would make it a structure-preserving method in the sense of \cite{HaiLubWan06}. Thus our motivations are close in spirit to for instance \cite{Fra06}\cite{RicLub08}. The kind of variational structure that underlies RC has become a governing design principle both for finite element methods and integration of ordinary differential equations \cite{CiaLio91}\cite{MarWes01}, so that insights from these mature fields could well inspire decisive improvements in RC. 

We are not aware of any stringent convergence results for RC, except those of \cite{CheMulSch84}. There, it is shown that for any given smooth metric, the Regge metrics interpolating it, have a curvature (defined by Regge calculus) which converges in the sense of measures, when the mesh width goes to $0$, to the curvature (defined in the standard way) of the given smooth metric. In numerical analysis this would be called a consistence result. In general, consistence is only a step towards proving convergence of a given numerical method. We also point out that the convergence of RC is discussed, in less stringent terms, in the physics literature, e.g. \cite{BreGen01} and references therein.

In \cite{Chr04M3AS} we related the space of Regge metrics to Whitney forms \cite{Wei52}\cite{Whi57}. As remarked in \cite{Bos88}, Whitney forms correspond to lowest order mixed finite elements \cite{RavTho77}\cite{Ned80}, the so-called edge and face elements, for which one has a relatively well developed convergence theory \cite{BreFor91}\cite{RobTho91}. More recently this analysis has been cast in the language of differential forms and related to Hodge theory \cite{Hip99}\cite{Chr07NM}\cite{ArnFalWin10}. We showed, in \cite{Chr04M3AS}, that there is a natural basis for the space of Regge metrics expressed in terms of Whitney forms and that second order differential operators restricted to Courant elements (continuous piecewise affine functions) are in one to one correspondence with linear forms on Regge metrics, edge elements and  Courant elements. This link integrates Regge calculus into the finite element framework. However we did not approach the question of curvature which is central to RC.

In this paper we further develop the theory of Regge elements. We first insert them in a complex of spaces equipped with densely defined interpolators providing commuting diagrams as in finite element exterior calculus \cite{ArnFalWin06}. The differential operators of this complex are those of linear elasticity. For the purposes of relativity, it appears that less regularity is required of the fields than for continuum mechanics, so that the discrete complex we obtain differs from those constructed in for instance \cite{ArnFalWin06IMA2}: the last two spaces in our complex consist of matrix and vector valued measures that cannot be represented by integrable functions. Next we provide results concerning the Regge action.

A priori it is not clear if RC should be considered a conforming or a non-conforming finite element method. Is the Regge action the restriction to Regge metrics of some extension by continuity of the Einstein-Hilbert action, to a large enough class of non-smooth metrics? One might compare with Wilson's lattice gauge theory discretization of the Yang-Mills equations \cite{Rot05}, where the discrete action is \emph{not} a simple restriction of the continuous one.  Indeed, seen as functionals, defined on the vector space of Lie-algebra valued (tensor product) Whitney forms, the Yang-Mills action is polynomial of order 4, whereas Wilson's action is transcendental, even in the case of Maxwell's equations \cite{ChrHal09BIT}.

Given a metric, the scalar curvature multiplied with the volume form provided by the metric, is a certain density on space-time depending non-linearly on the metric (and its derivatives). For a Regge metric, which has only partial continuity properties between simplexes, it is not clear which, if any, of the available expressions of the curvature in terms of the metric, make sense. Partial derivatives of discontinuous functions can be defined as distributions or currents, in the sense of Schwartz and de Rham, but for distributions, products are notoriously ill defined. In addition, if the background is only a piecewise affine manifold, the associated transitions between coordinate maps are Lipschitz but non-smooth, so that distribution theory seems inappropriate.

The arguments put forward by Regge to justify the definition of the Regge action are integral in nature. In RC there is a natural notion of parallel transport along paths that avoid the codimension-2 skeleton of the simplicial complex.  Around closed loops, this parallel transport behaves \emph{as if} there were a curvature, concentrated only on the codimension-2 skeleton, with the expression provided by RC in terms of deficit angles. See also the justifications provided in \cite{FriLee84}. To the author, it seems desirable that this \emph{ad hoc}  point of view, be related to the contemporary mathematical theory of partial differential equations. In this paper we provide results concerning linearization only, but holonomies will nevertheless play a pivotal role.

If we expand, as is done for instance in \cite{Hoo01}, the Einstein-Hilbert action in small perturbations around Minkowski space-time, the linear term is $0$ since the Minkowski metric solves the Einstein equations. The first non-trivial term is a quadratic form which we denote by $\calQ$. The Euler-Lagrange equations corresponding to finding critical points of $\calQ$, among ``all'' metrics, are nothing but the linearized Einstein equations, which describe the propagation of infinitesimal gravitational ripples on flat space-time. 

It appears\footnote{See the acknowledgement.} that the spatial part of $\calQ$ is associated with the $\curltcurl$ operator appearing in the linear elasticity complex \cite{ArnFalWin06IMA2}, where it encodes the Saint-Venant compatibility conditions. We show that this (spatial) quadratic form, defined a priori for smooth fields, has a natural extension to (spatial) Regge metrics. Moreover we show that this extension corresponds to the quadratic part of the Regge action, establishing that the first non-trivial terms (the second variations) of the Regge action and the Einstein-Hilbert action agree.

However, the natural Hilbert space on which the quadratic form is continuous (just) fails to contain the Regge elements! We argue that RC is a minimally non-conforming method. As a step towards an analysis of the convergence of numerical methods based on RC, we show that the eigenvalues for the $\curltcurl$ operator are well approximated with Regge elements. For an alternative convergence result concerning linearized RC, see \cite{BarWil88}.

The theory we develop is inspired by works on the eigenvalue problem for Maxwell's equations, \cite{Kik89}\cite{BofFerGasPer99}\cite{CaoFerRaf01} and also \cite{Chr07NM}\cite{ArnFalWin06}\cite{ChrWin10}. As for Maxwell's equations, the operator does not have a compact resolvent (due to the existence of an infinite dimensional kernel), so the basic theory \cite{BabOsb91} has to be amended. Indeed it has been shown that in this situation, stability is not sufficient to get eigenvalue convergence \cite{BofBreGas00}. At least two additional difficulties arise. First, linked to the above mentioned problem of hyperbolicity, is the fact that there are eigenvalues of arbitrary magnitude of both signs. One of the signs corresponds to modes that are excluded by the constraints in the continuous case. Due to the lack of sign, Cea type arguments valid for Maxwell's $\curl \curl$ operator (which is positive semi-definite), have to be replaced by inf sup conditions \cite{Bab71}\cite{Bre74}. Second is the (limit) non-conformity of the method. Central to the argument we develop is an analogue  for metrics of the Hodge decomposition of differential forms.

A number of interesting related results have been published while this paper was under review. We mention some, that come in addition to those already cited. RC has been described using dual tessellations \cite{DonMil08}, in a framework reminiscent of the discrete exterior calculus of \cite{DesHirLeoMar05}. In a similar vein of relating RC to notions of discrete mechanics, we also point out \cite{Ste09}, whereas \cite{Peu09} concerns a finite element point of view, as in \cite{Zum09}. Also of interest is \cite{BahDit10}. Numerical methods based on differential forms have been studied on manifolds \cite{HolSte10} and applied to general relativity in a simplified setting \cite{RicFra10}. Regge elements have been rediscovered as a tool for solving equations of elasticity \cite{CiaCia09}. Hodge decompositions of tensor fields, of the type used in this paper, have been studied in Lipschitz domains \cite{GeyKra09}. 

\paragraph{Layout.}
The paper is organized as follows. In section \ref{sec:reggeelt} we study Regge elements in finite element terms. We see what happens when we apply the Saint-Venant operator to them, and, based on the formula we obtain, insert them in a discrete elasticity complex. In section \ref{sec:lin} we relate the linearized Regge action to the $\curltcurl$ operator, showing that the second variations of the Regge action and the Einstein Hilbert action agree. In section \ref{sec:eigen} we study the discrete eigenvalue problem for the Saint-Venant operator on Regge elements. An abstract framework is introduced and then applied to our case.

\section{Regge elements and linear elasticity\label{sec:reggeelt}}

\paragraph{Basis and degrees of freedom.} 

We consider a space-slice $S$ which is a compact flat Riemannian 3-dimensional manifold without boundary. For definiteness we actually consider:
\begin{equation}
S = (\bbR/ l_1 \bbZ) \times (\bbR/ l_2 \bbZ) \times (\bbR/ l_3 \bbZ),
\end{equation}
for some positive reals $l_1, l_2, l_3$. On $S$ we have the Riemannian metric inherited from the standard Euclidean structure of $\bbR^3$. 

We put $\bbV= \bbR^3$ and let $\bbM$ be the space of $3\times 3$ real matrices. The subspace of $\bbM$ consisting of symmetric matrices is denoted $\bbS$, whereas that of antisymmetric matrices is denoted $\bbA$. Elements of $\bbV$ will be identified with $3 \times 1$ matrices called column vectors (and reals with $1 \times 1$ matrices). Thus the scalar product $v \cdot v'$ of two vectors $v,v' \in \bbV$, can also be written:
\begin{equation}
v \cdot v' = v^{\transp} v'.
\end{equation}

We let $C^\infty(S)$ denote the space of smooth real functions on $S$. The space of smooth vector fields on $S$ can be identified with $C^{\infty}(S) \otimes \bbV$. We regard the gradient of a function (at a point) to be a column vector, so that we have a map:
\begin{equation}
 \grad : C^\infty(S) \to C^{\infty}(S) \otimes \bbV.
\end{equation}
Likewise, $C^{\infty}(S) \otimes\bbS$ can be identified with the space of smooth symmetric $3 \times 3$ matrix fields.

We partition $S$ into tetrahedrons by a simplicial complex $\calT_h$. The set of $k$-dimensional simplexes in $\calT_h$ is denoted $\calT_h^k$. As is customary, the parameter $h$ denotes the largest diameter of a simplex of $\calT_h$. In \S \ref{sec:eigen}  we will be interested in convergence results as $h \to 0$, but until then our results concern a given $h$ representing a fixed mesh.

Regge metrics are symmetric matrix fields on $S$ that are piecewise constant with respect to $\calT_h$ and such that for any two tetrahedrons sharing a triangle as a common face, the tangential-tangential component of the metric is continuous across the face. Thus our metrics can be degenerate -- we do not impose any triangle inequalities, as would be necessary to define a distance from the metric. The continuity property imposed on Regge metrics can also be expressed by saying that the pullback of a metric, seen now as a bilinear form, to the interface between two tetrahedrons, is the same from both sides. This space of metrics has one degree of freedom per edge, which in the non-degenerate case can be taken to be the length (or length squared) of the edge, as defined by the metric. We proceed to give basic properties of this space, including our particular choice of basis and degrees of freedom.

In \cite{Chr04M3AS} we  related this space of metrics, which we denote by $X_h$,  to Whitney forms. For each vertex $x \in \calT^0_h$ let $\lambda_x$ denote the corresponding barycentric coordinate map. It is nothing but the continuous piecewise affine function taking the value $1$ at vertex $x$, and $0$ at other vertexes. Then the following family of metrics is a basis of $X_h$, indexed by edges $e \in \calT_h^1$:
\begin{equation}
\rho_e = 1/2 \ \big ( (\grad \lambda_{x_e})  (\grad \lambda_{y_e})^{\transp}  + (\grad \lambda_{y_e})  (\grad \lambda_{x_e})^{\transp} \big ),
\end{equation}
where the vertexes of the edge $e$ are denoted $x_e$ and $y_e$.

We define degrees of freedom as follows. For any edge $e \in \calT^1_h$ consider the linear form on smooth metrics:
\begin{equation}\label{eq:mue}
\mu_e : u \mapsto \int_0^1 (y_e - x_e)^{\transp} u(x_e + s (y_e - x_e)) (y_e - x_e) \rmd s. 
\end{equation}
One checks that on a tetrahedron these degrees of freedom are uni-solvent on the constant metrics. For two edges $e,e' \in \calT_h^1$ we have:
\begin{equation}
\mu_e(\rho_{e'}) = \left\{ 
\begin{array}{l l} 
0 & \textrm{ if } e \neq e',\\
1 & \textrm{ if } e = e'.
\end{array} \right.
\end{equation}
The degrees of freedom (\ref{eq:mue}) make sense for some non-smooth metrics as well, in particular elements of $X_h$. The interpolator associated with these degrees of freedom is the projection $I_h$ onto $X_h$, which to a symmetric matrix field $u$, associates the unique element $u_h \in X_h$ such that:
\begin{equation}
\forall e \in \calT^1_h \quad \mu_e(u_h) = \mu_e(u).
\end{equation}
The degrees of freedom do indeed guarantee tangential-tangential continuity of the interpolate.

\paragraph{Distributional Saint-Venant operator.} Recall that the $\curltcurl$ operator is defined on $3 \times 3$ matrix fields by taking first the $\curl$ of its lines, to obtain a new $3 \times 3$ matrix, then transposing and then taking once again the $\curl$ of its lines. If one starts with a symmetric matrix field, the result is a symmetric matrix field. We derive an expression for $\curltcurl u$ when $u \in X_h$. For this purpose we need some expressions concerning differential operators acting on distributions.

Let $T$ be a Lipschitz domain in $S$, with outward pointing unit normal $n \in \rmL^\infty(\partial T) \otimes \bbV$. Let $\delta_{\partial T}$ denote the Dirac surface measure on $\partial T$, defined by, for $\phi \in C^\infty(S)$:
\begin{equation}
\langle \delta_{\partial T}, \phi \rangle = \int_{\partial T} \phi|_{\partial T}.
\end{equation}

Let $u$ be the restriction to $T$, of a smooth scalar or vector field on $S$, which we extend by $0$ outside $T$. The restriction of $u$ to $\partial T$ (seen from the inside of $T$) is denoted $\gamma(u)$. By integration by parts in $T$, we have, as distributions on $S$:
\begin{align}
\grad u & = \grad_T u - \gamma(u) n \ \delta_{\partial T},\\
\curl u & = \curl_T u + \gamma(u) \times n \ \delta _{\partial T}, \label{eq:curludist}\\
\div u  & = \div_T u  - \gamma(u) \cdot n \ \delta _{\partial T}.
\end{align}
Here, a differential operator $\op$ on the left hand side is defined in the sense of distributions or currents on $S$, whereas on the right hand side, $\op_T$ denotes the corresponding operator defined classically  inside $T$.

Let $F$ be a (two-dimensional) domain inside a smooth oriented hyper-surface of $S$, with piecewise smooth (one-dimensional) boundary. The oriented unit normal on $F$ is denoted $n$ and the inward pointing unit normal of $\partial F$ inside the hyper-surface is denoted $m$. The Dirac surface measure on $F$ is denoted $\delta_F$, while the double layer distribution is denoted $\delta'_F$. The Dirac line measure on $\partial F$ is denoted $\delta_{\partial F}$. We have:
\begin{equation}
\grad \delta_F =n \delta'_F +  m \delta_{\partial F}.
\end{equation}
This formula can be applied as follows. Let $u$ be a smooth scalar or vector field on $S$. Recall that smooth functions  can be multiplied with distributions so that the product $u \delta_F$ is well-defined (it is a distribution with support on $F$). A Leibniz rule holds for such products. We have:
\begin{align}
\grad (u \delta_F) & = (\grad u) \delta_F + u n \delta'_F +  u m \delta_{\partial F},\\
\curl (u \delta_F) & = (\curl u) \delta_F - u \times n \delta'_F -  u\times m \delta_{\partial F}, \label{eq:curludelta}\\
\div ( u \delta_F) & = (\div u) \delta_F + u \cdot n \delta'_F +  u \cdot m \delta_{\partial F}.
\end{align}

For any $3$-vector $v \in \bbV$, let $\Skew v \in \bbA$ be the anti-symmetric $3 \times 3$ matrix defined by:
\begin{equation}
(\Skew v) v' = v \times v'.
\end{equation}

\begin{lemma}
Referring to Figure \ref{fig:sector}, consider a sector in $\bbV$ between two half-planes $F_0$ and $F_1$ originating from a common edge $E$ with unit tangent $t$. The unit outward-pointing normal on the planes is denoted $n$. The inward pointing normals to the edge in the planes are denoted $m_0$ and $m_1$. We let $n_i = \pm n$ be the normal to $F_i$ such that $(m_i,n_i,t)$ is oriented. Upon relabeling we may suppose $n_0 = -n$ and $n_1 = n$. 

The Dirac surface measure on the boundary $F = F_0 \cup F_1$ of the sector is denoted $\delta_F$ and the double-layer distribution $\delta'_F$. The Dirac line measure on the edge is denoted $\delta_E$. 

In the sector consider a constant metric $u$, extended by $0$ outside of it. Then:
\begin{align}
\curltcurl u =& (\Skew n) u (\Skew n) \delta_F' + \label{eq:sksk}\\
&   \big( (\Skew n_1) u (\Skew m_1) - (\Skew n_0) u (\Skew m_0)\big) \delta_E.
\end{align}
\end{lemma}

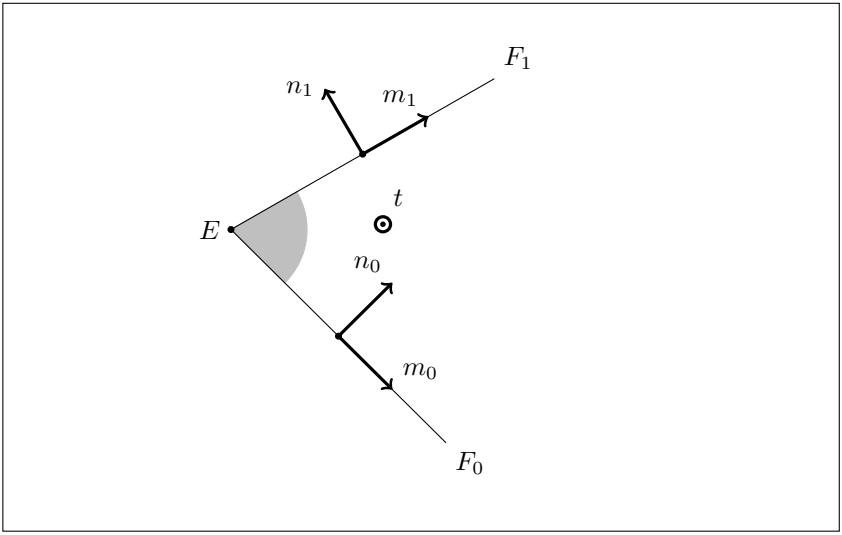
\begin{figure}[htbp]
\begin{center}
\begin{tikzpicture}

\clip[draw] (-1, 0) rectangle (10, 7);

\filldraw (2,4) [fill = lightgray, draw = lightgray] (2,4) -- ++(-45:1cm) arc(-45:30:1cm) -- cycle;

\draw (2, 4) -- +(-45:4cm) node[anchor = north west] {$F_0$};
\draw (2, 4) -- +(30:4cm) node[anchor = south west] {$F_1$};

\draw[->, very thick] (2,4) ++(-45:2cm) -- +(-45:1cm)   node[anchor = south west] {$m_0$};
\draw[->, very thick] (2,4) ++(-45:2cm) -- +(45:1cm) node[anchor = south east] {$n_0$};
\draw[very thick] (2, 4)++(-45:2cm) circle (0.025cm);

\draw[->, very thick] (2,4) ++(30:2cm) -- +(30:1cm) node[anchor = south east] {$m_1$};
\draw[->, very thick] (2,4) ++(30:2cm) -- +(120:1cm) node[anchor = east] {$n_1$};
\draw[very thick] (2, 4)++(30:2cm) circle (0.025cm);

\draw[very thick] (2, 4) circle (0.025cm) node[anchor = east]{$E$} ;

\draw[very thick] (2,4) ++(2:2cm) circle (0.1cm) circle (0.015cm) ++(0,0.1)node[anchor = south west]{$t$};

\end{tikzpicture}
\caption{View along the edge $E$ of a sector between half-planes $F_0$ and $F_1$. \label{fig:sector}}
\end{center}
\end{figure}

\begin{proof}
From formula (\ref{eq:curludist}) applied to each line of $u$ we deduce:
\begin{equation}
\curl u = u (\Skew n) \delta_F.
\end{equation}
Transposing we get:
\begin{align}
\transp \curl u & =- (\Skew n) u \delta_F,\\
&=  (\Skew n_0) u \delta_{F_0} - (\Skew n_1) u \delta_{F_1} .
\end{align}
We now apply (\ref{eq:curludelta})  to the lines of the two matrices on the right hand side and get the result.
\end{proof}

\begin{remark}
It is not clear from the above expression that $\curltcurl$ is a symmetric matrix field. However if we denote by $R_\rho$ the rotation around the vector $t$ by an angle $\rho$, there is an angle $\theta$ such that $R_\theta$ sends the basis $(m_0,n_0,t)$ to $(m_1,n_1,t)$. Put:
\begin{equation}
A(\rho) = (\Skew R_\rho n_0) u (\Skew R_\rho m_0),
\end{equation}
so that:
\begin{equation}
A(\theta)- A(0) = (\Skew n_1) u (\Skew m_1) - (\Skew n_0) u (\Skew m_0).
\end{equation}
Taking derivatives with respect to $\rho$ gives:
\begin{equation}
A'(\rho) = - (\Skew R_\rho m_0)u (\Skew R_\rho m_0) + (\Skew R_\rho n_0) u (\Skew R_\rho n_0),
\end{equation}
which is symmetric. Therefore its integral must be symmetric.
\end{remark}

\begin{proposition} \label{prop:exprctc} Consider now a simplicial complex $\calT_h$ in our domain $S$.

For any edge $e\in \calT^1_h$, let $\delta_e$ denote the Dirac line measure on $e$ and $t_e$ the unit oriented tangent vector along $e$. Let $f$ be a face having $e$ as an edge. We let $m_{ef}$ be the unit vector in the face $f$, orthogonal to the edge $e$ and pointing into the  face. We let $n_{ef}$ be the unit vector orthogonal to the face $f$ oriented such that $(m_{ef}, n_{ef}, t_e)$ is an oriented basis of $\bbV$ (the vector $n_{ef}$ depends on $e$ only for a sign). Let $\jump{u}_{ef}$ be the jump of $u$ across face $f$ in the order of $n_{ef}$.

We have:
\begin{equation}\label{eq:ctc1}
\curltcurl u = \sum_e \jump{u}_e t_e t_e^{\transp} \delta_e,
\end{equation}
where we sum over edges $e$ and put:
\begin{equation}\label{eq:ctc2}
 \jump{u}_e = \sum_f m_{ef}^{\transp}\jump{u}_{ef}n_{ef},
\end{equation}
where we sum over faces $f$ containing the edge $e$.
\end{proposition}

\begin{proof}
We use the preceding lemma. The double layer distributions cancel two by two, by the tangential-tangential continuity of $u$. We have:
\begin{equation}
\curltcurl u = \sum_e \sum_f (\Skew n_{ef}) \jump{u}_{ef} (\Skew m_{ef}) \delta_e.
\end{equation}
Remark that the lines of the matrix $(\Skew n_{ef}) \jump{u}_{ef}$ are proportional to $n_{ef}$ (by tangential-tangential continuity of $u$). In other words there is a vector $\alpha_{ef} $ such that:
\begin{equation}
(\Skew n_{ef}) \jump{u}_{ef}  = \alpha_{ef} n_{ef}^{\transp}.
\end{equation}
Then write:
\begin{equation}
(\Skew n_{ef}) \jump{u}_{ef} (\Skew m_{ef}) =\alpha_{ef}n_{ef}^{\transp} (\Skew m_{ef}) =  -\alpha_{ef}t_e^{\transp}.
\end{equation}  
This matrix has lines proportional to $t_e$. It follows that the lines of the matrix $\sum_f (\Skew n_{ef}) \jump{u}_{ef} (\Skew m_{ef})$ must be proportional to $t_e$. Since in addition this matrix is symmetric, it can be written:
\begin{equation}
\sum_f (\Skew n_{ef}) \jump{u}_{ef} (\Skew m_{ef}) = s_e t_e t_e^{\transp}.
\end{equation}
The scalar coefficient $s_e$ is determined by taking traces:
\begin{align}
s_e & = \tr \sum_f (\Skew n_{ef}) \jump{u}_{ef} (\Skew m_{ef}),\\
&= -\tr \sum_f  (\Skew n_{ef}) \jump{u}_{ef}  (t_e n_{ef}^{\transp} -  n_{ef} t_e^{\transp}),\\
& = -\sum_f  n_{ef}^{\transp} (\Skew n_{ef}) \jump{u}_{ef}  t_e -\sum_f   t_e^{\transp} (\Skew n_{ef}) \jump{u}_{ef}  n_{ef}   ,\\
&= \sum_f m_{ef}^{\transp}\jump{u}_{ef}n_{ef},
\end{align}
as announced.
\end{proof}
One can compare with the approach of \cite{HauKuhOrt07}, where tangential-tangential continuity is not enforced on the symmetric matrix fields. The discrete curvature they define is based on inter-element jumps, in an expression corresponding to the double layer distribution in (\ref{eq:sksk}).

\paragraph{A discrete elasticity complex.} The  space $X^h$ can be inserted in a complex of spaces $X^k_h$ with $0\leq k\leq 3$, each equipped with a densely defined interpolator $I^k_h$.

\describe{0} We let $X^0_h$ denote the space of continuous piecewise affine vector fields. It consists of the vector fields of the form:
\begin{equation}
v = \sum_{x \in \calT^0_h}  v_x \lambda_x,
\end{equation}
for all choices of vectors $v_x \in \bbV$ assigned to vertexes $x \in \calT^0_h$. We equip $X^0_h$ with the nodal interpolator $I^0_h$, which to any $v \in C^\infty(S) \otimes \bbV$ associates:
\begin{equation}
I^0_h v = \sum_{x \in \calT^0_h} v(x) \lambda_x.
\end{equation}

\describe{1} We put $X^1_h = X_h$, and equip it with the interpolator $I^1_h = I_h$.  For any element of $X^0_h$, its symmetrized gradient is piecewise constant and tangential-tangential continuous. In other words, the deformation operator, denoted $\defo$, induces a map $X^0_h \to X^1_h$. For our choice of bases of these spaces we notice that if $e$ is an edge with vertexes $x$ and $y$, we have for any $v \in \bbV$: 
\begin{equation}
\mu_e (\defo v \lambda_x) = (y-x)^{\transp} v.
\end{equation}

\describe{2} We let $X^2_h$ denote the space of matrix valued edge measures of the form:
\begin{equation}
u = \sum_{e \in \calT^1_h} u_e t_e t_e^{\transp} \delta_e,
\end{equation}
where  for each edge $e$, $u_e$ is a real number, $t_e$ is the unit oriented tangent vector to $e$ and $\delta_e$ is the Dirac line measure on $e$. Proposition \ref{prop:exprctc} shows that $\curltcurl$  induces a map $X^1_h \to X^2_h$. The standard $\rmL^2$ duality on matrix fields extends to a non-degenerate bilinear form on $X^2_h \times X^1_h$, which we denote by $\langle \cdot, \cdot \rangle$. Notice that, comparing with the definition of degrees of freedom $\mu_e$ on $X^1_h$ in equation (\ref{eq:mue}), we have for any $u \in X^1_h$:
\begin{align}
\langle t_e t_e^{\transp} \delta_e, u \rangle & = \int_e t_e^{\transp} u t_e, \\
& = \mu_e (u) / l_e, \label{eq:mudist}
\end{align}
where $l_e$ is the Euclidean length of the edge $e$. Therefore the interpolator $I^1_h$ deduced from the degrees of freedom $\mu_e$ satisfies for any  $u \in C^\infty(S) \otimes \bbS$:
\begin{equation}\label{eq:defi1}
\forall v \in X^2_h\quad \langle v, I^1_hu \rangle = \langle v, u \rangle.
\end{equation}

Just as the elements of $X^2_h$ act as degrees of freedom for $X^1_h$, the elements of $X^1_h$ can be used as degrees of freedom for $X^2_h$. The associated interpolator $I^2_h$ onto $X^2_h$ is determined by the property that for any $u\in \rmL^2(S) \otimes \bbS$, $I^2_hu$ satisfies:
\begin{equation}\label{eq:defi2}
\forall v \in X^1_h \quad \langle I^2_h u , v \rangle = \langle u , v \rangle.
\end{equation}
We remark that, by (\ref{eq:defi1}) and (\ref{eq:defi2}), for all $u \in \rmL^2(S)\otimes \bbS$ and all $v \in C^\infty(S) \otimes \bbS$:  
\begin{equation}
\langle I^2_h u, v \rangle = \langle I^2_h u, I^1_h v \rangle = \langle u, I^1_h v \rangle.
\end{equation}
In other words, $I^1_h$ and $I^2_h$ are adjoints of one another.

\describe{3} We let $X^3_h$ denote the space of vector vertex measures of the form:
\begin{equation}
u = \sum_{x \in \calT^0_h} u_x \delta_x,
\end{equation}
where for each vertex $x$, $ u_x\in \bbV$ is a vector and $\delta_x$ is the Dirac measure attached to $x$. The elements of $X^3_h$ act as natural degrees of freedom for $X^0_h$, yielding the nodal interpolator $I^0_h$ of $X^0_h$. The standard $\rmL^2$ duality on vector fields extends to a non-degenerate bilinear form on $X^3_h \times X^0_h$, denoted $\langle \cdot, \cdot \rangle$.
We define an interpolator $I^3_h$ onto $X^3_h$ by requiring, for any $u \in \rmL^2(S) \otimes \bbV$:
\begin{equation}
\forall v \in X^0_h \quad \langle I^3_h u , v \rangle = \langle u , v \rangle.
\end{equation}
As in the preceding case we remark that for all $u \in \rmL^2(S)\otimes \bbV$ and all smooth $v\in C^\infty (S) \otimes \bbV$:
\begin{equation}
\langle I^3_h u, v \rangle = \langle I^3_h u, I^0_h v \rangle = \langle u, I^0_h v \rangle.
\end{equation}
One also checks that for an edge $e$ with unit tangent $t_e$ and vertexes $x$ and $y$, such that $t_e = (y-x)/|y-x|$, we have:
\begin{equation}
\div (t_e t_e^{\transp} \delta_e) = t_e(\delta_y - \delta_x).
\end{equation}
This shows that the divergence operator induces a map $ X^2_h \to X^3_h$.

\vspace{1ex}

That concludes the list of spaces and operators we need to form our diagram. The following theorem summarizes some of the above remarks, and relates the continuous elasticity complex \cite{ArnFalWin06IMA2} to a discrete one. 

\begin{theorem}\label{theo:comm}
We have a commuting diagram of spaces:
\begin{equation}
\xymatrix{
C^{\infty}(S) \otimes \bbV \ar[r]^{\defo}\ar[d]^{I^0_h} & C^{\infty}(S) \otimes\bbS \ar[rr]^{ \curltcurl}\ar[d]^{I^1_h}& & C^{\infty}(S) \otimes \bbS \ar[r]^\div \ar[d]^{I^2_h} & C^{\infty}(S) \otimes\bbV \ar[d]^{I^3_h}\\
X^0_h \ar[r]^{\defo} & X^1_h \ar[rr]^{\curltcurl}& & X^2_h \ar[r]^{\div} & X^3_h
}
\end{equation}
On the lower row the linear operators are defined in the sense of distributions.
\end{theorem}

\begin{proof} Only commutativity remains to be proved.

(i) For any $u \in C^\infty(S) \otimes \bbV$ and any edge $e \in \calT^1_h$ with vertexes $x, y$, we have:
\begin{equation}
\mu_e( \defo u) =   t_e^{\transp} (u(y) - u(x)). 
\end{equation}
We deduce :
\begin{equation}
\mu_e( \defo u) = \mu_e (\defo I^0_h u).
\end{equation}
Commutation of the first square follows.

(ii) For any $u \in C^{\infty}(S) \otimes\bbS$ and any $v \in X^1_h$, we have:
\begin{align}
\langle I^2_h \curltcurl u , v \rangle & =\langle \curltcurl  u, v \rangle,\\
& =\langle \curltcurl  v, u \rangle, \\
&= \langle \curltcurl v, I^1_h u \rangle, \\
&= \langle \curltcurl I^1_h u,  v \rangle,
\end{align}
so that:
\begin{equation}
I^2_h \curltcurl u = \curltcurl I^1_h u.
\end{equation}
This proves commutation of the middle square. 

(iii) Commutation of the last square follows from (i) by duality.
\end{proof}

The discrete fields considered here have less regularity than those defined in \cite{ArnAwaWin08}, which are all at least square integrable, throughout the complex. Spaces of matrix fields adapted to a second order differential operator have also been considered in \cite{SchSin07}. In their case, the differential operator is ``$\div \div$'' (extracting first the divergence, per line say, of the matrix field and then the divergence of the obtained vector field), for which normal-normal continuity of the matrix fields is the natural analogue of our tangential-tangential continuity. Note also that the two-dimensional case of the ``Regge complex'' was studied in \cite{CiaCia09} (independently of any reference to Regge calculus).

\section{Linearizing the Regge action \label{sec:lin}}

The setting and notations are as in the preceding section. Our aim is to compute the second variation of the Regge action and relate it to the previously exhibited Saint-Venant operator. 

We consider first what happens around a single edge. Fix an oriented line (edge) in $\bbV$ with unit tangent $t$. The set-up is similar to the one of Figure \ref{fig:sector}. Originating from this edge are half-planes (faces) indexed by a cyclic parameter $f$ and ordered counter-clockwise. Thus the face coming immediately after $f$ is denoted $f+1$. The sector between faces $f$ and $f+1$ is indexed by $f+1/2$. Let $m_f$ be the oriented unit length vector in the half-plane  $f$ which is orthogonal to $t$. Let $n_f$ be the normal to half-plane $f$, so that $(m_f,n_f,t)$ is an oriented orthonormal basis of $\bbV$. 

Each sector between half-planes $f$ and $f+1$ is equipped with a constant metric $u_{f+1/2} = u_{f+1/2}(\epsilon)$ depending on a small parameter $\epsilon$. We suppose that $u(\epsilon)$ is continuous across the half-plane $f$ in the tangential-tangential directions. We suppose that $u_{f + 1/2}(0)$ is the canonical Euclidean metric on $\bbV$, and we will be particularly interested in the first derivative of $u_{f + 1/2}(\epsilon)$ with respect to $\epsilon$ at $\epsilon = 0$, which we denote by $u_{f + 1/2}'$.

To ease notations, the dependence upon $\epsilon$ will be implicit in what follows. The derivative of a function $\psi :\epsilon \mapsto \psi(\epsilon)$ at $\epsilon$ is denoted $\psi'(\epsilon)$. Unless otherwise specified we only differentiate at $\epsilon = 0$ and therefore write $\psi' = \psi'(0)$.

We want to parallel transport a vector around the edge, by a path going once around it and with respect to the metric $u(\epsilon)$. In each sector parallel transport is trivial, but from one sector to another, say from $f- 1/2$ to $f+1/2$ we denote by $T_f$ the matrix of the parallel transport in the basis $(m_f, n_f,t)$. It is defined as follows. The oriented unit normal to face $f$ with respect to $u_{f-1/2}$ is denoted $k_f^-$, that with respect to $u_{f+1/2}$ is denoted $k_f^+$. The operator $T_f$ maps the basis $(m_f,k_f^-,t)$ to $(m_f,k_f^+,t)$ . We denote by $R_{f}^{g}$ the matrix of the identity operator from basis $(m_f,n_f,t)$ to basis $(m_{g},n_{g},t)$.

The holonomy from sector $f-1/2$ to itself, in the basis $(m_f, n_f,t)$ is defined to be:
\begin{equation}\label{eq:hol}
E_{f-1/2} = R_{f-1}^{f}T_{f-1} \cdots R_{g}^{g+1}T_{g} \cdots R_f^{f+1}T_f.
\end{equation}
Remark that $T_f$ is an isometry from the metric $u_{f-1/2}$ to the metric $u_{f+1/2}$. Consequently $E_{f-1/2}$ is an isometry with respect to $u_{f-1/2}$. In the plane orthogonal to $t$ it is a simple rotation.
The angle of this rotation does not depend on $f$ and is the \emph{deficit angle} associated with the edge $t$. We denote it by  $\theta$.

\begin{proposition}
The derivative of the deficit angle $\theta$ at $\epsilon = 0$ is given as a sum of jumps:
\begin{equation}\label{eq:sumjumps}
\theta'= 1/2 \sum_f m_f^{\transp} (u'_{f+1/2} -u'_{f-1/2}) n_f.
\end{equation}
\end{proposition}
\begin{proof}
Let $\tilde m_f$ be the oriented unit normal to $t$ in face $f$, with respect to $u_{f-1/2}$ or equivalently $u_{f+1/2}$. Let $P_f$ be the matrix of the identity from basis $(m_f,n_f,t)$ to $(\tilde m_f, k_f^-,t)$. Since $(\tilde m_f, k_f^-)$ is an orthonormal oriented basis of the plane orthogonal to $t$, with respect to the metric induced by $u_{f-1/2}$ we have:
\begin{equation}
E_{f-1/2} = P_f^{-1} \left (
\begin{array}{rrr}
\cos \theta & -\sin \theta & 0\\
\sin \theta &  \cos \theta & 0\\
0 & 0& 1
\end{array}
\right)
P_f.
\end{equation}
Differentiating this expression at $\epsilon = 0$, using that $P_f(0)$ is the identity matrix,  we get:
\begin{equation}
E_{f-1/2}' = \left(
\begin{array}{rrr}
0^{\phantom'} & -\theta'&0\\
\theta' & 0^{\phantom'}&0\\
0^{\phantom'} &0^{\phantom'} &0
\end{array}
\right).
\end{equation}
Differentiating (\ref{eq:hol}) at $\epsilon=0$ we obtain, since $T_f(0)$ is the identity matrix: 
\begin{equation}
E_{f-1/2}'= R_{f-1}^{f} T'_{f-1} R_f^{f-1}+  \cdots+ R_{g}^{f} T'_{g} R_f^{g}  +\cdots + T'_f.
\end{equation}
Let $J$ be the canonical skew matrix:
\begin{equation}
J= \left(\begin{array}{rrr}
0 & -1 & 0\\
1 & 0 & 0\\
0 & 0 & 0
\end{array}\right).
\end{equation} 
It commutes with all $R_f^{g}$. We have:
\begin{align}
\theta'&= -1/2 \tr(JE_{f-1/2}^{\prime})\\
&= -1/2 \sum_g \tr(J T'_g).\label{eq:thetasumtr}
\end{align}
We determine the terms in this sum. Let $M_g^\pm$ be the matrix in $(m_g,n_g,t)$ of the operator sending $(m_g,n_g, t)$ to $(m_g, k_g^\pm, t)$. We have:
\begin{equation}
T_g= M_g^+ (M_g^{-})^{-1}.
\end{equation}
Differentiating with respect to $\epsilon$ at $\epsilon=0$ we obtain, since $M_g^\pm$ is the identity matrix at $\epsilon = 0$:
\begin{equation}
T_g'= (M_g^+)' - (M_g^{-})'.
\end{equation}
Define reals $\alpha_g^\pm, \beta_g^\pm, \gamma_g^\pm$ by:
\begin{equation}
M_g^\pm = \left (
\begin{array}{rrr}
1 & \alpha_g^\pm & 0\\
0 & \beta_g^\pm & 0\\
0 & \gamma_g^\pm & 1
\end{array}
\right). \label{eq:defmgpm}
\end{equation}
Then: 
\begin{equation}\label{eq:tralpha}
\tr(J T'_g) = (\alpha_g^+)' - (\alpha_g^-)'.
\end{equation}
Differentiating the identity :
\begin{equation}
m_g^{\transp} u_{g + 1/2} k_g^+ = 0,
\end{equation}
at $\epsilon = 0$ yields:
\begin{equation}
m_g^{\transp}u'_{g + 1/2}n_g + m_g^{\transp}(k_g^+)' = 0.
\end{equation}
Since also, by the definition (\ref{eq:defmgpm}):
\begin{equation}
\alpha_g^+ = m_g^{\transp} k_g^+,
\end{equation}
it follows that:
\begin{equation}
(\alpha_g^+)' = - m_g^{\transp}u'_{g+1/2}n_g.
\end{equation}
Similarly we have:
\begin{equation}
(\alpha_g^-)' = - m_g^{\transp}u'_{g-1/2}n_g.
\end{equation}
Insert these two expressions in (\ref{eq:tralpha})  and combine with (\ref{eq:thetasumtr}) to get the proposition.
\end{proof}

We consider now the general case of a simplicial complex in $S$. For a Regge metric $u \in X_h$ its Regge action $\calR(u)$ is defined as follows. For any edge $e$ its length, as determined by $ u$, is denoted $l_e(u)$ and its deficit angle $\theta_e(u)$. Then, summing over edges, one defines \cite{Reg61}: 
\begin{equation}
\calR(u) = \sum_e \theta_e(u) l_e(u).
\end{equation}

\begin{proposition} Let Regge metrics $u(\epsilon) \in X_h$ depend smoothly on a small real parameter $\epsilon$. 
We use the notations:
\begin{align}
\calR(\epsilon) &= \calR(u(\epsilon)),\
\theta_e(\epsilon) = \theta_e(u(\epsilon)),\
l_e(\epsilon)  = l_e(u(\epsilon)).
\end{align} 

We suppose that $u(0)$ is the constant Euclidean metric. 
Then we have:
\begin{equation}\label{eq:linregge}
\calR(\epsilon)= \epsilon^2/4 \sum_e \theta_e'(0) l_e(0)^{-1} \mu_e(u'(0)) + \calO(\epsilon^3).
\end{equation}
\end{proposition}
\begin{proof}
We have:
\begin{equation}
\theta_e(0)= 0.
\end{equation}
Consequently:
\begin{align}
\calR(\epsilon) & = \sum_e (\epsilon \theta_e'(0) + \epsilon^2/2 \ \theta_e''(0)) \ (l_e(0) + \epsilon l_e'(0)) \ + \ \calO(\epsilon^3)\\
& = \epsilon \sum_e \theta_e'(0)l_e(0) \ + \ \epsilon^2 \sum_e (\theta_e'(0) l_e'(0) + 1/2 \ \theta_e''(0)l_e(0)) \ + \ \calO(\epsilon^3).\label{eq:expr}
\end{align}
It is a remarkable property of Regge calculus, proved in \cite{Reg61}, that for any $\epsilon$: 
\begin{equation}
\sum_e \theta_e'(\epsilon)l_e(\epsilon) = 0.
\end{equation}
This handles the first term in (\ref{eq:expr}). Differentiating we get in addition:
\begin{equation}
\sum_e \theta_e''(0) l_e(0) + \theta_e'(0) l_e'(0)= 0.
\end{equation}
Inserting this in the second term of (\ref{eq:expr}) we get:
\begin{equation}
\calR(\epsilon) = \epsilon^2/2 \sum_e \theta_e'(0) l_e'(0)+ \calO(\epsilon^3).
\end{equation}
Finally consider an edge $e$ and denote its extremities by $x$ and $y$. We have $l_e(0) = | y - x|$ and $t_e = (y - x)/|y - x|$. We can write:
\begin{align}
l_e(\epsilon) & = ((y - x)^{\transp} u(\epsilon)(y - x))^{1/2}\\
& = (|y - x|^2 + \epsilon  (y - x)^{\transp}u'(0)(y - x) + \calO (\epsilon^2))^{1/2}\\
& = l_e(0)(1 + \epsilon/2 \ t_e^{\transp} u'(0)t_e + \calO (\epsilon^2),
\end{align}
so that:
\begin{align}
l_e'(0) &  = 1/2\  l_e(0)t_e^{\transp} u'(0)t_e,\\
& = 1/2\  l_e(0)^{-1} \mu_e(u'(0)).
\end{align}
This completes the proof.
\end{proof}

Now insert expressions (\ref{eq:sumjumps}) and (\ref{eq:mudist}) in (\ref{eq:linregge}) and compare with the expressions (\ref{eq:ctc1}) and (\ref{eq:ctc2}) for the $\curltcurl$ operator. We conclude:
\begin{theorem}
With the above notations we have the expansion:
\begin{equation}
\calR(u(\epsilon))= \epsilon^2/8 \ \langle \curltcurl u'(0), u'(0) \rangle + \calO(\epsilon^3).
\end{equation}
\end{theorem}
It can be checked that $\curltcurl$ is also the operator appearing in the spatial part of the second variation of the Einstein Hilbert action. It suffices to write $\curltcurl$ in coordinates (index notation) and compare with the expression for the second variation provided in, for instance, \cite{Hoo01}.

\section{Eigenvalue approximation \label{sec:eigen}}

\paragraph{An abstract setting.}
In this paragraph we develop a convergence theory for some non-conforming eigenvalue approximations. It is a generalization of the point of view detailed in \cite{ChrWin10} for conforming approximations of semi-definite operators. Some of our arguments are closer to a discontinuous Galerkin way of thinking \cite{BufPer06} but we also need to cover the case of operators with spectrum in both ends of the real axis, so that coercivity should be replaced by $\inf \sup$ conditions \cite{Bab71}.

Let $O$ be a Hilbert space, with scalar product $\langle \cdot, \cdot \rangle$. The associated norm is denoted $| \cdot |$. Let $X$ be a dense subspace of $O$, which is also a Hilbert space such that the injection $X \to O$ is continuous. The norm of $X$ is denoted $\| \cdot \|$. We suppose that we have a symmetric continuous bilinear form $a$ on $X$. We do not require $a$ to be semi-definite. 
We are interested in the eigenvalue problem, to find $u \in X$ and $\lambda \in \bbR$ such that:
\begin{equation}\label{eq:eig}
\forall v \in X \quad a(u,v) = \lambda \langle u, v \rangle.
\end{equation}
We define:
\begin{align}
W &= \{ u \in X \ : \ \forall v \in X \quad a(u,v) = 0 \},\\
V &= \{ u \in X \ : \ \forall v \in W \quad \langle u, v \rangle = 0 \}.
\end{align}
These are closed subspaces of $X$ and we have a direct sum decomposition:
\begin{equation}
X = V \oplus W.
\end{equation}
We suppose that $W$ is closed in $O$ and that the injection $V \to O$ is compact. Let $\overline V$ be the closure of $V$ in $O$. Let $P$ be the orthogonal projection in $O$ with range $\overline V$ and kernel $W$.

We suppose that the map $X \to X^\star$, $u \mapsto a(u, \cdot)$ determines an isomorphism $V \to V^\star$, equivalently:
\begin{equation}
\infsup{u\in V}{v\in V}{|a(u,v)|}{\| u\| \, \| v\|} > 0.
\end{equation}

The dual of $X$ with $O$ as pivot space is denoted $X'$. Notice that we distinguish it from the space of continuous linear forms on $X$ which is denoted $X^\star$. For instance $O$ is a dense subspace of $X'$ but not of $X^\star$. The map $u \mapsto \langle u, \cdot \rangle$ defines an isomorphism $X' \to X^\star$. The notation $Y'$ will be used later, with the same meaning, for other dense subspaces $Y$ of $O$. 

Define a bounded operator $K : X' \to X$ as follows. To any $u \in X'$ associate $Ku=v\in V$ such that:
\begin{equation}
\forall w \in V \quad a(v, w) = \langle u, w \rangle.
\end{equation}
As an operator $O \to O$, $K$ is compact and selfadjoint. This guarantees that there is an orthonormal basis of $O$ consisting of eigenvectors of $K$. The non-zero eigenvalues are inverses of those of (\ref{eq:eig}), with corresponding eigenspaces.

We now turn to the the approximation of the eigenproblem (\ref{eq:eig}) by a Galerkin method.

Suppose that $(X_n)$ is a sequence of finite-dimensional subspaces of $O$, and that for each $n \in \bbN$ we have a symmetric bilinear form $a_n$ on $X_n$. We solve the eigenvalue problems: find $u \in X_n$ and $\lambda \in \bbR$ such that:
\begin{equation}\label{eq:galeig}
\forall v \in X_n \quad a_n(u,v) = \lambda \langle u, v \rangle.
\end{equation}
Decompose as before:
\begin{align}
W_n &= \{ u \in X_n \ : \ \forall v \in X_n \quad a(u,v) = 0 \},\\
V_n &= \{ u \in X_n \ : \ \forall v \in W_n \quad \langle u, v \rangle = 0 \}.
\end{align}
We have a direct sum decomposition:
\begin{equation}
X_n = V_n \oplus W_n.
\end{equation}
Define also $K_n: O \to O$ as follows. To any $u \in O$ associate $K_n u = v \in V_n$ such that:
\begin{equation}
\forall w \in V \quad a_n(v, w) = \langle u, w \rangle.
\end{equation}
Since $K_n$ has finite rank it is compact. It is also selfadjoint. Non-zero eigenvalues of $K_n$ are the inverses of those of (\ref{eq:galeig}), with corresponding eigenspaces.

For the eigenpairs of $K_n$ to converge to those of $K$ in a natural sense, the following should be achieved \cite{Cha83}:
\begin{equation}\label{eq:conveig}
\| K- K_n \|_{O \to O} \to 0.
\end{equation}
Our aim now is to devise sufficient conditions for this to hold. For insights into some difficulties that arise in the context of discretizations of the form (\ref{eq:galeig}), when $a$ has infinite dimensional kernel, see in particular \cite{BofBreGas00}.

For two real sequences $(a_n)$ and $(b_n)$, estimates of the form, there exists $C >0$, independent of the sequences, such that for all $n$:
\begin{equation}
a_n \leq C b_n,
\end{equation}
will be written:
\begin{equation}
a_n \cleq b_n \myor b_n \cgeq a_n.
\end{equation}
We also use the notation:
\begin{equation}
a_n \ceq b_n,
\end{equation}
when:
\begin{equation}
a_n \cleq b_n \myand b_n \cleq a_n.
\end{equation}

For the applications we have in mind, $X_n$ is not a subspace of $X$, but we can weaken the norm of $X$ to obtain a space containing $X_n$, without loosing essential compactness properties. Precise statements follow. We suppose that we have two more Hilbert spaces $X^-$ and $X^+$ with inclusions:
\begin{equation}
X^+ \subset X \subset X^- \subset O,
\end{equation}
which are continuous and have dense range. The norms of $X^+$ and $X^-$ are denoted $\|\cdot \|_+$ and $\| \cdot \|_-$ respectively. We suppose that $a$ is continuous on $X^+ \times X^-$. We define:
\begin{align}
W^+ &= \{ u \in X^+ \ : \ \forall v \in X^- \quad a(u,v) = 0 \},\\
W^- &= \{ u \in X^- \ : \ \forall v \in X^+ \quad a(u,v) = 0 \},\\
V^+ &= \{ u \in X^+ \ : \ \forall v \in W^+ \quad \langle u, v \rangle = 0 \},\\
V^- &= \{ u \in X^- \ : \ \forall v \in W^- \quad \langle u, v \rangle = 0 \}.
\end{align}
We have direct sum decompositions:
\begin{equation}
X^+ = V^+ \oplus W^+ \myand X^- = V^- \oplus W^-.
\end{equation}
We {suppose that $W^+$ is closed in $O$}, which yields:
\begin{equation}
W^+ = W = W^-.
\end{equation}
We suppose also that that we have the {inf-sup conditions}:
\begin{equation}\label{eq:isap}
\infsup{u\in V^+}{v\in V^-}{|a(u,v)|}{\| u\|_+ \, \| v\|_-} > 0 
. 
\end{equation}
It follows that $K$ is bounded $X^{-\prime} \to X^+$.

We also suppose that {the injection $V^- \to O$ is compact}. By duality the injection $\overline V \to X^{-\prime}$ is compact, so that, by composition, $K$ is compact as an operator $O \to X_+$.

Concerning our spaces $X_n$ we suppose that they are equipped with projections $Q_n : O \to X_n$ which are uniformly bounded $O \to O$ and map $W$ to $W_n$. Moreover we suppose that $\delta(W_n, W) \to 0$, the gap being calculated with the norm of $O$. Recall the definition:
\begin{equation}
\delta(W_n, W) = \sup_{u \in W_n} \inf_{v \in W} |u -v|/ |u| .
\end{equation}

We impose of course that for any $u \in O$ there is a sequence $u_n \in X_n$ such that $u_n \to u$ in $O$. We then have $Q_n u \to u $ in $O$.
\begin{lemma}\label{lem:split} Consider sequences $u_n = v_n + w_n$ with $v_n \in V_n$ and $w_n \in W_n$. Suppose that $u_n \to u$ in $O$, and decompose $u = v +w$  with $v \in \overline V$ and $w \in W$. Then $v_n \to v$ and $w_n \to w$  in $O$.
\end{lemma}
\begin{proof}
We have:
\begin{equation}
 0 \leftarrow |u - u_n|^2 = |v - v_n|^2 + |w-w_n|^2 - 2 \langle v, w_n \rangle - 2 \langle v_n, w \rangle.
\end{equation}
Considering the right hand side we have:
\begin{align}
|\langle v, w_n \rangle | & = |\langle v, w_n - (I - P) w_n\rangle| \leq |v| \, |w_n| \, \delta(W_n, W)\\
&\leq |v| \, |u_n| \, \delta(W_n, W) \to 0,  
\end{align}
and also:
\begin{align}
|\langle v_n, w \rangle | &= |\langle v_n, w - Q_n w \rangle|\leq |v_n| \ |w - Q_n w|\\
& \leq  |u_n| \ |w - Q_n w| \to 0.
\end{align}
This gives the lemma.
\end{proof}

The spaces $X_n$ are not necessarily subspaces of $X$, but always of $X^-$. Since $W$ is closed in $X^-$, the following proposition states that, as $n \to \infty$,  $\delta(V_n, V^-) \to 0$, the gap being calculated with the norm of $X^-$.
\begin{proposition}\label{prop:gap}
There is a sequence $\epsilon_n \to 0$ such that for all $u \in V_n$:
\begin{equation}\label{eq:gap}
| u - P u| \leq \epsilon_n \|u\|_-.
\end{equation}
\end{proposition}
\begin{proof}
For any $u \in V_n$ we have:
\begin{equation}
u - Q_nPu = Q_n(u-Pu) \in W_n.
\end{equation}
We write:
\begin{align}
|Pu - Q_n Pu|^2 &= |Q_n (u-Pu) - (u-Pu)|^2\\
& = |Q_n (u - Pu)|^2 + |u - Pu|^2 - 2 \langle Q_n(u-Pu), Pu \rangle,
\end{align}
from which it follows that:
\begin{equation}
|u - Pu|^2 \leq |Pu - Q_nPu|^2 + 2 \delta(W_n, W)\, |Q_n(u - Pu)|\, |P u|.
\end{equation}
Recall the uniform boundedness of the $Q_n : O \to O$. The first term on the right hand side is handled by the fact that $P$ is compact as a map $X^- \to O$. The second term is handled by the gap property of $W_n$.
\end{proof}

The forms $a_n$ are required to be consistent with $a$ in the following sense. For all $u \in X^+$ there is a sequence $u_n \in X_n$ such that $u_n \to u$ in $O$ and:
\begin{equation}\label{eq:conan}
\lim_{n \to \infty} \sup_{v \in X_n}\frac{|a(u, v) - a_n(u_n, v)|}{\| v \|_-} = 0.
\end{equation}
Moreover we suppose that we have the following weak inf-sup condition, uniform in $n$:
\begin{equation}\label{eq:isan}
1 \cleq \infsup{u \in V_n}{v\in V_n}{|a_n(u,v)|}{|u|\, \|v\|_-}. 
\end{equation}

\begin{proposition}
Under the above circumstances we have discrete eigenpair convergence in the sense of (\ref{eq:conveig}).
\end{proposition}

\begin{proof}
(i) Define $P_n : X^+ \to X_n$ as follows. To any $u\in X^+$ associate $P_nu = v\in V_n$ such that:
\begin{equation}
\forall w \in V_n \quad a_n(v, w)= a(v,w).
\end{equation}
It follows from (\ref{eq:isan}) that $P_n$ is uniformly bounded $X^+ \to O$. To get pointwise convergence, pick $u \in X^+$. Choose $u_n \in X_n$ converging to $u$ in $O$, from the consistence hypothesis (\ref{eq:conan}). Decompose $u_n = v_n + w_n$ with $v_n \in V_n$ and $w_n \in W_n$. By Lemma \ref{lem:split}, we have $v_n \to Pu$ in $O$. We also have:
\begin{align}
|P_n u - v_n| &\cleq \sup_{w \in V_n}\frac{|a_n(P_n u - v_n, w)|}{\|w\|_-} = \sup_{w \in V_n}\frac{|a(u, w) - a_n(u_n, w)|}{\|w\|_-} \to 0.
\end{align}
We deduce $P_n u \to P u$ in $O$.

We remark that for $u \in \overline V$ we have:
\begin{equation}
\forall v \in V_n \quad a_n(P_n K u, v) = a(Ku, v) = \langle u,v \rangle = a_n(K_n u, v),
\end{equation}
so that:
\begin{equation}
P_n K u = K_n u.
\end{equation}
Moreover, as already pointed out, $K$ is compact as an operator $O \to X^+$.

Combining these remarks we get:
\begin{equation}
\| K- K_n \|_{\overline V \to O} \to 0.
\end{equation}

(ii) For $u \in W$ we have:
\begin{equation}
\forall v \in V_n \quad a_n(K_n u, v) = \langle u,v \rangle = \langle u , v - Pv \rangle.
\end{equation}
Using (\ref{eq:isan}) and Proposition \ref{prop:gap}, we can write:
\begin{align}
| K_n u | & \cleq \sup_{v \in V_n}\frac{|a_n(K_nu, v)|}{\| v \|_-}=  \sup_{v \in V_n}\frac{\langle u, v - Pv\rangle}{\| v \|_-}\\
&\cleq \epsilon_n |u|.
\end{align}
This gives:
\begin{equation}
\| K_n \|_{W \to O} \to 0.
\end{equation}
Since $K$ is zero on $W$ this concludes the proof.
\end{proof}

\paragraph{Application to Regge calculus.}
We first define some function spaces that will allow us to use the preceding setting.

We define $O  = \rmL^2(S) \otimes \bbS$, equipped with the canonical scalar product, and let $a$ be the bilinear form on symmetric matrix fields, which, at least for smooth ones is defined by:
\begin{equation}
a(u,v)= \langle \curltcurl u, v \rangle.
\end{equation}
We define for any $\alpha \in [-1, 1]$:
\begin{equation}
X^\alpha = \{ u \in O \ : \ \curltcurl u \in \rmH^{-1+ \alpha}(S) \otimes \bbS \}.
\end{equation}
Fix now $\alpha \in ]0, 1/2[$ and put:
\begin{align}
X^+ & = X^{\alpha},\\
X^{\phantom{ +}} & = X^0,\\
X^- &= X^{-\alpha}.
\end{align}

\begin{proposition}
The required hypotheses are satisfied:
\begin{itemize}
\item $a$ defines a continuous bilinear form on $X$,
\item the kernel $W$ is closed in $O$,
\item $V$ is compactly embedded in $O$,
\item $a$ is invertible on $V \times V$.
\end{itemize}
Moreover:
\begin{itemize}
\item $a$ defines a continuous bilinear form on $X^+ \times X^-$,
\item the kernel $W^+$ is closed in $O$,
\item $V^-$ is compactly embedded in $O$,
\item the inf sup condition (\ref{eq:isap}) holds.
\end{itemize}
\end{proposition}
\begin{proof}
We base the proof on Fourier analysis. For any $\xi \in \bbV$ we denote by $F(\xi): \bbV \to \bbC$ the associated Fourier mode:
\begin{equation}
F(\xi) : x \mapsto \exp(i \xi \cdot x).
\end{equation}
For any $a \in \bbS$ we have:
\begin{equation}
\curltcurl (a F(\xi)) = \Skew(\xi) a \Skew(\xi) \, F(\xi).
\end{equation}
For $\xi \neq 0$ choose two normalized vectors $\xi_1, \xi_2 \in \bbV$ such that $(\xi, \xi_1, \xi_2)$ is an oriented orthogonal basis of $\bbV$. We make some remarks:

First, $a \in \bbS$ satisfies $a \xi = 0$ iff $a$ is a linear combination of the three matrices:
\begin{align}
\sigma_1(\xi) &= \xi_1 \xi_1^{\transp} + \xi_2 \xi_2^{\transp},\\
\sigma_2(\xi) &= \xi_1 \xi_2^{\transp} + \xi_2 \xi_1^{\transp},\\
\sigma_3(\xi) &= \xi_1 \xi_1^{\transp} - \xi_2 \xi_2^{\transp}.
\end{align}
We compute:
\begin{align}
\Skew(\xi) \sigma_1(\xi) \Skew(\xi) &= -|\xi|^2 \sigma_1(\xi),\\
\Skew(\xi) \sigma_2(\xi) \Skew(\xi) &= \phantom{+}|\xi|^2 \sigma_2(\xi),\\
\Skew(\xi) \sigma_3(\xi) \Skew(\xi) &= -|\xi|^2 \sigma_3(\xi).
\end{align}

Second, $a \in \bbS$ is orthogonal to these three matrices iff:
\begin{equation}
\Skew(\xi) a \Skew(\xi) = 0.
\end{equation}
In fact this occurs iff $a$ can be written, for some (uniquely determined) $b \in \bbV$:
\begin{equation}
a =  b\xi^{\transp} + \xi b^{\transp}.
\end{equation}

Since we restricted attention to periodic boundary conditions, $\rmL^2(S)\otimes \bbC$ is spanned by Fourier modes:
\begin{equation}
F(k) \ : \ k \in (2 \pi/ l_1) \bbZ \times (2 \pi/ l_2) \bbZ \times (2 \pi/ l_3) \bbZ.
\end{equation}
The preceding remarks then provide a Hilbertian basis of $O$ consisting of eigenvectors for $\curltcurl$. Taking into account the characterization of Sobolev spaces in terms of Fourier series, all the claimed results follow.
\end{proof}

We consider a quasi-uniform sequence of simplical meshes $\calT_h$ of $S$, with mesh width $h \to 0$. Attached to $\calT_h$ we denote as before by $X_h= X^1_h$ the space of Regge metrics, and equip it with the bilinear form $a_h$ induced by $a$, defined for $u,v \in X^1_h$ by:
\begin{equation}
a_h(u,v) = \langle \curltcurl u, v \rangle,
\end{equation}
with respect to the canonical duality pairing $\langle \cdot, \cdot \rangle$  on $X^2_h \times X^1_h$.

The elements of $X^2_h$ are not in $\rmH^{-1}(S) \otimes \bbS$ since traces on edges are not well defined in $\rmH^1(S)$, therefore $X_h$ is not a subspace of $X$. One might therefore be reluctant to call $a_h$ the restriction of $a$. However, consistence of $a_h$ with $a$ is not in doubt. We remark also that $W_h$ is a subspace of $W$ (since $X^1_h$ acts as degrees of freedom for $X^2_h$) so that $\delta(W_h, W) = 0$. We will use that traces on edges are well defined $\rmH^{1+\alpha}(S) \to \rmL^2(e)$, from which it follows that $X_h$ is in $X^-$.

\begin{proposition} The spaces $X_h$ can be equipped with projections $Q_h$ that are uniformly bounded $O \to O$ and map $W$ to $W_h$.
\end{proposition}

\begin{proof}
We use the technique of  \cite{ArnFalWin06}, which was based on the earlier  works \cite{Sch08}\cite{Chr07NM}.

Let $\phi: \bbR^3 \to \bbR$ be smooth, non-negative with support in the unit ball and integral $1$. Define its scaling by:
\begin{equation}
\phi_{h}(x) = h^{-3} \phi(h^{-1} x).
\end{equation}
Let $R^\epsilon_h$ be regularization by convolution by $\phi_{\epsilon h}$. Remark that it commutes with constant coefficient differential operators. Composing with the already defined interpolator $I_h$ we get an operator $I_h R^\epsilon_h : O \to X_h$. We can fix a small $\epsilon$ such that for all $h$:
\begin{equation} 
\| (\id - I_h R_h)|_{X_h} \|_{O \to O}\leq 1/2.
\end{equation}
We may then put:
\begin{equation}
Q_h = (I_h R_h|_{X_h})^{-1}I_h R_h,
\end{equation}
to get the required projection.
\end{proof}

\begin{proposition} \label{prop:wis} We have an estimate, uniform in $h$: 
\begin{equation}
\forall u \in X_h \quad \|P u\|_- \cleq  \sup_{v \in X_h} \frac{|a_h(u,v)|}{|v|}.
\end{equation}
\end{proposition}
\begin{proof}
We first evaluate the norm of the restriction operator from a tetrahedron $T\in \calT_h$ to an edge $e$ in norms $\rmH^{1+\alpha}(T) \to \rmL^2(e)$.

We let $\hat T$ be a reference tetrahedron of diameter $1$. We suppose that the scaling $x \mapsto hx$, maps the tetrahedron $\hat T$ to $ T$ and the edge $\hat e$ to $ e$. For  a $k$-multilinear form $u$ on $T$, its pullback $\hat u$ to $\hat T$ satisfies, for $x \in \hat{T}$:
\begin{equation}
\hat u (x)= h^k u (hx).
\end{equation}
With this definition, we remark that $\nabla$ commutes with the pullback by scaling maps. Specializing to the case $k=2$ we have the scaling estimates:
\begin{align}
 \| u \|_{\rmL^2(T)} & = h^{-1/2} \| \hat u (x) \|_{\rmL^2(T)}, \\
 \| \nabla u \|_{\rmL^2(T)} & = h^{-3/2} \| \nabla \hat u (x) \|_{\rmL^2(T)}, \\
 | u |_{\rmH^\alpha(T)} & = h^{-3/2 - \alpha } | \hat u (x) |_{\rmH^\alpha(T)}.
\end{align}
From this we deduce:
\begin{align}
| \ts \int_e t_e^{\transp} u  t_e | & = h^{-1} |\ts \int_{\hat{e}} t_{\hat{e}}^{\transp}  u   t_{\hat{e}}|\\
 & \cleq  h^{-1} \| \hat{u} \|_{\rmH^{1+ \alpha}(\hat{T})}\\
 & \cleq  h^{-1/2} \| u \|_{\rmH^{1+ \alpha}(T)}.
\end{align}
Therefore, summing over all tetrahedrons:
\begin{equation}
\big(\sum_e ( \ts \int_e t_e^{\transp} u t_e)^2 \big )^{1/2} \cleq h^{-1/2} \| u \|_{\rmH^{1+ \alpha}(S)}.
\end{equation}
By duality we get, for any family of reals $v_e \in \bbR$ attached to the edges $e$:
\begin{equation}
\| \sum_e v_e t_e t_e^{\transp} \delta_e \|_{\rmH^{-1- \alpha}(S)} \cleq h^{-1/2} (\sum_e v_e^2)^{1/2}.
\end{equation}
By scaling, for such a family $(v_e)$ we also have:
\begin{equation}
\| \sum_e v_e \rho_e \|_{\rmL^2(S)} \ceq h^{-1/2} (\sum_e v_e^2)^{1/2}.
\end{equation}
Pick now $u \in X_h$ and define the family $(v_e)$ by:
\begin{equation}
 \curltcurl u = \sum_e v_e t_e t_e^{\transp} \delta_e.
\end{equation}
We define $v \in X_h$ by:
\begin{equation}
 v = \sum_e v_e \rho_e ,
\end{equation}
and can now write:
\begin{align}
\frac{ | \langle \curltcurl u, v \rangle | }{ | v | }& = \frac{ | \sum_e v_e^2  \int_e t_e^{\transp} \rho_e t_e| }{ | v | }\\
& \cgeq h^{-1/2} (\sum_e v_e^2)^{1/2}\\ 
& \cgeq  \| \curltcurl u\|_{\rmH^{-1- \alpha}(S)}.
\end{align}
This completes the proof.
\end{proof}

\begin{corollary}\label{cor:infsupd} There is a lower bound uniform in $h$: 
\begin{equation}
\inf_{u \in V_h} \sup_{v \in V_h} \frac{|a_h(u,v)|}{|u| \|v\|_{-}} = \inf_{u \in V_h} \sup_{v \in V_h} \frac{|a(u,v)|}{\| u\|_- |v|} \cgeq 1.
\end{equation}
\end{corollary}
\begin{proof}
To obtain the uniform lower bound for the second term remark first that Proposition \ref{prop:gap} yields for $u_h \in V_h$:
\begin{equation}
\| u\|_- \approx \| P u\|_-.
\end{equation}
Then apply Proposition \ref{prop:wis}.

The equality reflects that a map and its adjoint have the same norm. 
\end{proof}

Therefore the abstract setting can be applied to prove (on toruses):
\begin{theorem}
Linearized Regge calculus yields a convergent method to approximate the eigenpairs of the Saint-Venant operator.
\end{theorem}


\section*{Acknowledgments}
I thank Ragnar Winther for stimulating discussions and Douglas N. Arnold and Ragnar Winther for generously sharing their notes on the linearized Einstein equations with me. In particular this is where I learned of the relationship of these equations to elasticity.

This work, conducted as part of the award ``Numerical analysis and simulations of geometric wave equations''  made under the European Heads of Research Councils and European Science Foundation EURYI (European Young Investigator) Awards scheme, was supported by funds from the 
Participating Organizations of EURYI and the EC Sixth Framework Program.

\bibliography{../Bibliography/alexandria}{}
\bibliographystyle{plain}

\end{document}